\documentclass[acmsmall,screen]{acmart}

\AtBeginDocument{%
  \providecommand\BibTeX{{%
    \normalfont B\kern-0.5em{\scshape i\kern-0.25em b}\kern-0.8em\TeX}}}

\setcopyright{acmcopyright}
\copyrightyear{2020}
\acmYear{2020}
\acmDOI{0}

\acmJournal{TOMACS}
\acmVolume{0}
\acmNumber{0}
\acmArticle{0}
\acmMonth{0}

\def\N{\mathbb{N}}
\def\P{\mathbb{P}}
\def\bQ{\mathbb{Q}}

\def\R{\mathbb{R}}

\def\Prob{\mathbb{P}}
\def\E{\mathbb{E}}

\def\ac{\mathcal{AC}}

\def\bR{\mathbb{R}}

\def\bx{\mathbf{x}}
\def\by{\mathbf{y}}
\def\bz{\mathbf{z}}
\def\be{\mathbf{e}}

\newcommand{\re}{\mathcal{H}}
\newcommand{\Var}{\textrm{Var}}

\newcommand{\Mane}{\textrm{Ma\~{n}\'{e}}}
\def\gradx{\nabla _{\mathbf{x}}}

\newcommand{\PN}[1]{{\color{black}#1}}
\newcommand{\rev}[1]{{\color{black}#1}}

\begin{document}

\title[Importance sampling for a simple Markovian intensity model]{Importance sampling for a simple Markovian intensity model using subsolutions}

\author{Boualem Djehiche}
\email{boualem@kth.se}
\author{Henrik Hult}
\email{hult@kth.se}
\author{Pierre Nyquist}
\email{pierren@kth.se}
\affiliation{%
  \institution{Department of Mathematics, KTH Royal Institute of Technology}
  \city{Stockholm}
  \country{Sweden}
  \postcode{100 44}
}

\renewcommand{\shortauthors}{Djehiche, Hult and Nyquist}

\begin{abstract}
This paper considers importance sampling for estimation of rare-event probabilities in a specific collection of Markovian jump processes \PN{ used for e.g.\ modelling of credit risk. Previous attempts at designing importance sampling algorithms have resulted in poor performance and the main contribution of the paper is the design of efficient importance sampling algorithms using subsolutions.} \PN{The dynamics of the jump processes causes the corresponding Hamilton-Jacobi equations to have an intricate state-dependence, which makes the design of efficient algorithms difficult.} We provide theoretical results that quantify the performance of importance sampling algorithms in general and construct asymptotically optimal algorithms for some examples. The computational gain compared to standard Monte Carlo is illustrated by numerical examples.

\end{abstract}




\maketitle

\section{Introduction}
\PN{In this paper we analyse and develop Monte Carlo methods for rare-event estimation in certain Markovian intensity models, a class of stochastic jump models where the jump-intensities are Markovian with respect to deterministic functions of the current state of the entire system. Such models are prevalent in e.g.\ operations research—particularly in the financial context, see e.g.\ \cite{BGG15} and the references therein, and for queueing models \cite{asmussen2003}—stochastic chemical kinetics \cite{AndersonKurtz2011, vanKampen81}, and population models \cite{Jager2013, MeleardVillemonais12}. For non-trivial jump-intensities, explicit computations become intractable and there is a need for efficient computational methods.}

\PN{We \PN{focus} our study on a particular class of models, previously used for e.g.\ credit risk modelling, for which the construction of efficient computational schemes turns out to be very challenging. Specifically, we consider the difficult task of designing efficient importance sampling algorithms for Markovian intensity models of the form used in \cite{CarmonaCrepey}. In \cite{CarmonaCrepey, Carmona2} the authors study different types of Monte Carlo methods for estimating rare-event probabilities, corresponding to large portfolio losses, in two models for credit risk. In \cite{CarmonaCrepey} both importance sampling and interacting particle systems are used to estimate the probability of large portfolio losses before a fixed time horizon and it is shown that both schemes result in poor performance for different instances of the underlying model. }

\PN{Importance sampling is one of the most successful approaches to rare-event sampling and has been used in a wide array of areas in applied probability \cite{asmussen2007, RubinoTuffin}. In recent years, two main systematic approaches have emerged for the design of provably efficient methods: one based on constructing appropriate Lyapunov functions, developed mainly by Blanchet, Glynn and co-authors —see e.g.\ \cite{BlanchetGlynn, BlanchetLiu, BlanchetLi11, BlanchetGlynnLeder12, BlanchetHultLeder13}—and one based on linking importance sampling to appropriate (sub)solutions of Hamilton-Jacobi equations, developed by Dupuis, Wang and co-authors—see e.g.\ \cite{BD19, DupuisWang, DupuisWangSubsol, DupuisWang2005, DupuisWangLeder_wslqIS, DupSezerWang2007,DupuisWangJackson, DupSpilWang2012, DupuisLederWang07}. 

In this work, we use the subsolution approach to address the problem studied in \cite{CarmonaCrepey}. For the model in \cite{CarmonaCrepey} we obtain asymptotically optimal importance sampling algorithms and, for a natural generalisation of that model, we show how the subsolution approach can be used to construct efficient algorithms, validated by their performance in numerical experiments. To the best of our knowledge, this is the first work to apply the subsolution approach to obtain efficient algorithms in the setting of pure-jump processes with jump rates that have a non-zero gradient on a set of positive measure. This is different from, for example, queueing models, where importance sampling has been used extensively, the qualitative difference being that in the setting considered here, affine functions of the state of the process will not produce efficient algorithms. Rather, here the gradient of the subsolution must also be state-dependent.}

\PN{The modelling of credit risk and defaults in large portfolios has seen much work in recent years, in addition to \cite{CarmonaCrepey, Carmona2}, see for example \cite{Spil1, Spil2, Spil3, BGG15} and references therein. In those works, the authors consider rather complex models for the stochastic default intensity, meant to capture properties observed in the market, study the behavior of defaults as the size of the portfolio goes to infinity and consider affine point process models in this setting. In \cite{Spil1} it is emphasised that standard Monte Carlo is typically slow for large portfolios and long time horizons (hence their desire to develop new methods). Understanding how to design efficient Monte Carlo methods, in this case importance sampling, even for rather simple models is a valuable step towards constructing fast and accurate methods for more involved systems. 
For a general overview of Monte Carlo methods used in financial engineering see \cite{Glasserman2004}; examples of the use of importance sampling can be found in \cite{Glasserman2002, GlassermanLi, GuasoniRobertson2008}. In the more general setting of Markov jump processes, other common Monte Carlo methods include splitting, the cross-entropy method and genealogical methods, see e.g.\ \cite{asmussen2007, RubinsteinKroese04, DelMoralGarnier}.}

We note that although the motivation for this particular model comes from credit risk, the problem under consideration fits into the more general context of Monte Carlo methods for \PN{Markov jump processes} with mean-field characteristics, as a large part of the analysis does not depend on the specific form of the jump intensities. 



The remainder of the paper is organized as follows. In Section \ref{sec:Model} we introduce the relevant Markovian intensity model for credit risk, associated stochastic processes and probabilities of interest. Section \ref{sec:LD} reviews large deviation results for the type of Markov processes under consideration. In Section \ref{sec:IS} we review the basics of importance sampling for processes of the type described in Section \ref{sec:Model}, including the relevant measure of efficiency, and describe how efficient algorithms are linked to subsolutions of Hamilton-Jacobi equations. The main result of the Section, Theorem \ref{thm:asymptOpt}, quantifies the performance of algorithms constructed from subsolutions. In Section \ref{sec:Subsol} we construct efficient algorithms in the multi-dimensional setting. Lastly, in Section \ref{sec:Numerics} we present numerical experiments  that illustrate the performance of the proposed importance sampling algorithms. For completeness, Appendix \ref{sec:Isaacs} contains a formal derivation of the Hamilton-Jacobi equation of Section \ref{sec:IS}. 

\PN{\subsection{Notation}
\label{sec:notation}
The following notation is used. Elements in $\bR ^d$ are denoted by bold font, e.g. $\bx, \by, \bz$, whereas $x,y,z$ denote elements of $\bR$. With some abuse of notation, $0$ denotes the zero element in $\bR ^d$ for any $d\geq 1$. For $j=1, \dots, n$, $\be _j$ denotes the $d$-dimensional vector with a $1$ in the $j$th entry and all remaning entries set to 0: $\be _1 = (1, 0, \dots, 0)$, $\be _2 = (0, 1, 0,\dots, 0)$ and so on.  For a set $A \in \bR ^d$ and $n \in \mathbb{N}$, $nA$ denotes the set $\{ n\mathbf{x} : \ \bx \in A\}$. For an open set $A$ in $\mathbb{R} ^d$, $C^1(A)$ is the space of continuous real functions on $A$ with continuous partial derivatives. For a compact set $K$, $C^1(K)$ refers to functions that are $C^1(A)$ on some open neighbourhood $A$ of $K$. For $T\in [0,\infty)$, $\ac ([0,T]; \R^d)$ denotes the set of absolutely continuous functions $\psi \colon [0,T] \to \R^{d}$ and $\mathcal{D}([0,\infty); \R^{d})$ denotes the set of c\`adl\`ag functions from $[0,\infty)$ to $\R ^d$. For a function $f: \bR ^d \to \bR$, $\nabla f$ denotes the gradient of $f$. The partial derivative with respect to one variable is denote by subscript: $f_{x_i} = \frac{\partial}{\partial x_i} f$. Similarly, for a function $f: \bR \times \bR ^d \to \bR$, the gradient of $f$ with respect to the $\bR^d$-valued argument $\bx$ is denoted by $\nabla _\bx f = (f_{x_1}, \dots, f_{x_d})$. For functions $f: \bR \to \bR$ the derivative is simply denoted $f'$.}

\section{Model and problem formulation}
\label{sec:Model}
\PN{The class of models of interest is a subset of continuous-time pure jump Markov processes on $\bR ^d$, for some $d\geq1$. More specifically, in this paper we restrict to pure-birth processes with an upper limit of $nw_{j} \in \mathbb{N}$ births in the $j$th component, $j=1, \dots, d$, for $w_{j} > 0$ such that $nw_j$ is integer and $w_{1}+ \dots + w_{d} = 1$. Such a process can be viewed as modelling a population of $n\in\N$ individuals divided into $d$ groups, with $nw_{j} \in \mathbb{N}$ individuals in the $j$th group, counting the number of times a specific event (corresponding to a jump/birth) occurs in each group; in the credit risk model in \cite{CarmonaCrepey}, mentioned in the previous section, the ``event'' is the default of an obligor and groups correspond to obligors with the same credit rating}. For notational convenience, \PN{we define $D \subset \bR ^d$ to be the Cartesian product
\begin{align*}
D = [0,w_1] \times \dots \times [0,w_d].
\end{align*}}
\PN{The sets $D$ and $nD = [0, nw_1] \times \dots \times [0, nw_d]$} will act as state spaces for the stochastic processes we consider. \PN{In this paper we focus on jump intensities of the form  
\begin{align}
\label{eq:Model}
	\lambda_{j}(\bx) = a_j(w_{j} - x_{j}) e^{b\langle 1,\bx\rangle}, \ j=1,\dots d, \ \rev{\bx \in D},
\end{align}
for $a_1,\dots ,a_d$ and $b$ in $\R _+$, where $\langle \cdot,\cdot \rangle$ denotes the standard scalar product in $\R^d$. For each $n$, $\{Q^{n}(t); t \geq 0 \}$, $Q^{n}(0) = \mathbf{0} \in \bR ^d$, denotes a $d$-dimensional continuous-time pure jump Markov process on $nD$ with infinitesimal generator \rev{$L^n$ defined as follows: for $\mathbf{q} \in nD$, so that $\mathbf{q} = n\bx$ for some $\bx \in D$,
\begin{align*}
L^{n}f(\mathbf{q}) = n\sum_{j=1}^{d} \lambda_{j}(\bx) [f(\mathbf{q}+\be_{j}) - f(\mathbf{q})],
\end{align*}
for $f$ in some suitable class of functions}. The process $Q^{n}_{j}(t)$ represents the number of jumps in the $j$th group up to time $t$. The form \eqref{eq:Model} is motivated by \cite{CarmonaCrepey}, where it is used for modelling credit risk. Therein the process $Q^n$ counts the number of defaults in $d$ subgroups of obligors, out of a portfolio of $n$ obligors. The vector $\mathbf{a} = (a_1,\dots,a_d)$ characterises the default intensities in the $d$ different groups, whereas $b$ determines the contagion effect of the total number of defaults on the entire portfolio. The model \eqref{eq:Model} is a generalisation of the one used in the examples in \cite{CarmonaCrepey}: here different groups are allowed to have different intensities $a_i$. In \cite{CarmonaCrepey}, the authors hint at a model of the form \eqref{eq:Model}, with intensities differing between groups, but never explicitly state or consider any such examples. A similar collection of models is also considered in \cite{BGG15}. Therein the dependence structure does not allow for the type of contagion effect caused by the exponential term in \eqref{eq:Model}, however they consider more general jump-diffusion models. 
}

We are interested in studying the probability of \PN{ the sum of the components of $Q^n$} exceeding some (high) threshold before \PN{a fixed time horizon $T$}, which is also the focus of \cite{CarmonaCrepey}. More precisely, for some $z \in (0,1)$, we study the probability $p_n$ given by
\begin{align*}
	p_n = \Prob \biggl( \sum _{j=1} ^d Q _j ^n(T) \geq nz   \biggr) = \Prob \biggl( \sum _{j=1} ^d Q _j ^n (t) \geq nz, \textrm{ for some } t \leq T \biggr).
\end{align*}
The second equality follows from the fact that $Q^n$ is non-decreasing. For large $n$ the event that $Q^n$ exceeds $nz$ before time $T$ is a rare event, i.e., the probability $p_n$ will be small. For such choices \PN{of} $z$ standard Monte Carlo will be inefficient for estimating $p_n$ and the goal of the paper is to construct efficient importance sampling algorithms for this task. The construction of such algorithms becomes challenging when the intensities $\lambda_j$ are state-dependent, \PN{particularly when there is dependence between the different components as in \eqref{eq:Model} (sometimes referred to as cross-excitation of the components)}, and a detailed analysis is needed to find an appropriate sampling distribution. In the context of credit risk the described problem amounts to studying the probability that the proportion of defaults in the portfolio exceeds $z$ before time $T$.

\PN{Importance sampling amounts to choosing a new set of jump intensities $\bar \lambda _j$, $j=1, \dots, d$, as described in Section \ref{sec:ISmodel}, and consider the corresponding jump process $\bar X ^n$. In order to choose the $\bar \lambda _j$s so that the resulting algorithm becomes efficient, asymptotic results as $n$ goes to infinity serve as a guide}. For this, denote by $\{X^{n}(t); t \geq 0\}$ the scaled process
\begin{align}
\label{eq:defXn}
	X^{n}(t) = \frac{1}{n}Q^{n}(t), 
\end{align}
also known as the density dependent population process; see e.g.\ Chapter 11 in \cite{EthierKurtz86}
The probability $p_n$ can then be expressed in terms of the scaled process $X^n$,
\begin{align}
\label{eq:probPn}
	p_n = \Prob \biggl( \sum _{j=1} ^d X_j ^n(T) \geq z \biggr) =  \Prob \biggl( \sum _{j=1} ^d X _j ^n (t) \geq z \textrm{ for some } t \leq T \biggr),
\end{align}
and large deviation asymptotics can be employed for this probability (i.e., for the process $X^n$) to aid in the choice \PN{of $\bar \lambda _j$s} and thus sampling distribution.



\section{Large deviations for the sequence of scaled jump-processes}
\label{sec:LD}
For each $n$, the process $\{ X^{n} (t); t \geq0 \}$ in \eqref{eq:defXn} is a continuous-time pure jump Markov process with infinitesimal generator $A ^n$ defined by 
\begin{align*}
	A^{n}f(\bx) = n\sum_{j=1}^{d} \lambda_{j}(\bx) [f(\bx+\be_{j}/n) - f(\bx)],
\end{align*}
for $f$ in some suitable class of functions, and takes values (with probability one) in $\mathcal{D}([0,\infty); \R^{d})$. The generator $A^n$ has an associated (scaled) Hamiltonian, $H^n$, defined by
\begin{align*}
	H^n f(x) = \frac{1}{n}e^{-nf(x)} A^n e^{nf(\bx)} = \sum_{j=1} ^d \lambda_j(\bx) \bigl( e^{n(f(\bx+\be_j/n)-f(\bx))} - 1\bigr).
\end{align*}
For example, if the function $f$ is $C^1$ and, for $\alpha \in \R ^d$,  the sum $\sum _j \lambda_j (\boldsymbol{\mathbf{x}}) e^{\langle \alpha, \be_j \rangle}$ is finite, it holds that
\begin{align*}
	\lim _{n \to \infty} H^n f(\bx) = \sum _{j=1} ^d \lambda _j (\bx) \bigl( e^{\langle \nabla f(\bx), \be_j \rangle} -1 \bigr).
\end{align*}
Define the function $H: D \times \R^{d} \to \R$ by
\begin{align}
\label{eq:Hamiltonian}
H(\bx,\alpha) = \sum_{j=1}^{d} \lambda_{j}(\bx)\bigl(e^{\langle \alpha, \be_j \rangle}-1\bigr),
\end{align}
and let $L$ be the convex conjugate of $H$,
\begin{align*}
	L(\bx,\beta) = \sup_{ \alpha \in \R^{d}} \Big[\langle \alpha,\beta\rangle - H(\bx,\alpha)\Big]. 
\end{align*}
A straightforward calculation gives the explicit form of $L$, 
\begin{align}\label{eq:Lagrangian}
L(\bx,\beta) = \langle \beta, \log \frac{\beta}{\lambda(\bx)} \rangle - \langle \beta - \lambda(\bx), 1 \rangle, 
\end{align}
where $\beta / \lambda (\bx)$ denotes component-wise division; \PN{the functions $H$ and $L$ are referred to as the Hamiltonian and the Lagrangian, respectively.}

\PN{Due to the assumptions on the \rev{jump intensities $\lambda _j$—which define the jump rates $r^n$ of the $\{ X^n\}$ processes, see Section \ref{sec:ISmodel}—}for each $T < \infty$, the sequence $\{ X^n \}$ satisfies Laplace principle on the sample path level; see \cite{FengKurtz, ShwartzWeiss}.} 
\begin{theorem}
\label{thm:LD}
The sequence $\{X^{n}(t); t \in [0,T]\}$ satisfies the Laplace principle with rate function $I_{\bx}$ given by
 \begin{align*}
	I_{\bx}(\psi) = \begin{cases} \int_{0}^{T} L(\psi(t), \psi'(t)) dt, & \psi \in \ac ([0,T]; \R^d), \ \textrm{non-decreasing and } \psi(0)=\bx, \\ \infty, & \textrm{otherwise}. \end{cases}
\end{align*}
That is, for any $\bx \in \R ^d$ and bounded, continuous function $h: \mathcal{D}([0,T]; \R^{d}) \to \R$, 
\begin{align*}
	\lim_{n\to \infty}
 \frac{1}{n} \log E_{\bx}[\exp\{-nh(X^{n})\}] = - \inf _{\rev{\psi}} \{I_{\bx}(\psi)+h(\psi)\},
 \end{align*}
  \rev{ where the infimum is over $\psi \in \ac ([0,T]; \R^d)$, non-decreasing and $\psi(0)=\bx$.}
 
\end{theorem}

We end this section by hinting at how the large deviation principle connects to the design of efficient simulation algorithms. Let $D_z$ be the set
\begin{align*}
D_z = \biggl\{ \by = (y_1, \dots, y_d) \in D \colon \sum _{j=1} ^d y_j \geq z \biggr\},
\end{align*}
and define the function $U \colon [0,T] \times D \to [0,\infty]$ as a conditional version of the rate function:
\begin{align}
\label{eq:LDvar}
	U(t,\bx) = \inf _{\psi} \biggl\{ \int_t^T L(\psi(s), \psi'(s)) ds: \psi(t) = \bx, \ \psi(T) \in D_z \biggr\},
\end{align}
where the infimum is over $\psi \in \ac ([0,T]; \R ^d)$ that are non-negative and non-decreasing (in each component). For each pair $(t,\bx)$, $U(t,\bx)$ is interpreted as the large deviation rate function associated with  the probability of reaching the set $ D_z$ before time $T$, when starting in state $\bx$ at time $t$. 
According to Theorem \ref{thm:LD}, the convex conjugate $L$ of $H$ acts as the local rate function for the sequence $\{ X^n \}$. This conjugacy between $L$ and $H$ provides the connection to a Hamilton-Jacobi equation, which can be used for designing efficient simulation algorithms; see Section \ref{sec:Subsol} and Appendix \ref{sec:Isaacs}.

\section{Importance sampling}
\label{sec:IS}
In this section we first review the basics of importance sampling and describe the importance sampling estimator of \eqref{eq:probPn} when using an alternative stochastic kernel (corresponding to an alternative sampling distribution). We then introduce a Hamilton-Jacobi equation related to the dynamics of the model of Section \ref{sec:Model} and recall the notion of subsolutions to this type of partial differential equation (PDE). \PN{In Section \ref{sec:HJperf} we prove the main result of this section, Theorem \ref{thm:asymptOpt}, which characterises the performance of importance sampling algorithms based on subsolutions, by linking the relative error to the initial value of the subsolution.}

\subsection{Basics of importance sampling}
Importance sampling is the method to simulate a system under different dynamics, i.e.\ probability distribution, than in the original model. In the present setting the task is to estimate the probability $p_n = \Prob (X^n(T) \in D_z)$, where $\Prob$ describes the original dynamics for the process $X^n$. To perform importance sampling, consider different dynamics and the associated probability measure $\bar{ \bQ} ^n$ such that $\Prob \ll \bar{\bQ} ^n$; \PN{ it is enough for the absolute continuity to hold on a sub-$\sigma$-algebra that contains an appropriate part of the state space, e.g. in this case the set $D_z$}. One sample of the importance sampling estimator, denoted by $\widehat p _n$, is the indicator function of the event times the Radon-Nikodym derivative associated with the change of measure from $\Prob$ to $\bar{ \bQ} ^n$:
\begin{align*}
	\widehat p _n = I\{ X ^n (T) \in D_z \} \frac{d\Prob}{d \bar{ \bQ }^n},
\end{align*}
where $ X ^n$ now has dynamics according to $\bar{ \bQ} ^n$. Including the Radon-Nikodym derivative ensures that $\widehat p _n$ is an unbiased estimator of $p_n$:
\begin{align*}
	E _{\bar {\bQ} ^n} [\widehat p _n] = E _{\bar{\bQ} ^n}\biggl[ I\{ X ^n (T) \in D_z \} \frac{d\Prob}{d \bar{ \bQ }^n} \biggr] = E _{\Prob} [I\{ X ^n (T) \in D_z \}] = p_n.
\end{align*}
Choosing a suitable alternative measure $\bar{ \bQ} ^n$ requires a measure of efficiency. \PN{The standard measure of efficiency for unbiased estimators is the relative error,
\begin{align}
\label{eq:RE}
	\textrm{RE} (\widehat p_n) = \frac{\sqrt{\Var (\widehat p_n)}}{p_n},
\end{align}
where a smaller relative error corresponds to a more efficient algorithm. By considering the square of $\textrm{RE} (\widehat p_n)$ and writing out the definition of the variance, we see that minimising the relative error amounts to minimising the second moment $E _{\bar{\bQ} ^n} [\widehat p_n ^2]$. The aim is therefore to choose a sampling distribution that minimises the second moment whilst still being feasible to implement (cf.\ the optimal zero-variance change of measure \cite{asmussen2007}).

How small can we hope for the second moment to be? 
By Jensen's inequality, $E _{\bar{\bQ}^n} [\widehat p _n ^2] \geq p_n ^2$ and using the large deviation principle of Theorem \ref{thm:LD} we have
\begin{align*}
	\liminf _{n \to \infty} \frac{1}{n} \log E _{\bar{\bQ}^n} [\widehat p _n ^2 ]\geq 2 \liminf _{n \to \infty} \frac{1}{n} \log p_n \geq -2U(0,0),
\end{align*} 
where $U$ is defined in \eqref{eq:LDvar}. This lower bound for the logarithmic asymptotics of $\widehat p _n ^2$ holds true for any $\bar{\bQ}^n$. A given $\bar{\bQ} ^n$ is said to be \emph{asymptotically optimal} if the corresponding upper bound holds as well, that is if
\begin{align*}
	\limsup _{n \to \infty} \frac{1}{n} \log E_{\bar{\bQ}^n} [\widehat p _n ^2] \leq -2 U(0,0).
\end{align*}
This is the notion of optimality we are concerned with in this paper. For the upcoming analysis it is useful to note that the second moment of $\widehat p_n$ under $\bar{\bQ} ^n$ is equal to the first moment of $\widehat p _n$ under $\Prob$,
\begin{align*}
	E _{\bar{\bQ} ^n} \left[ \widehat p _n ^2 \right] &= E _{\bar{\bQ}^n} \left[ I\{ X ^n (T) \in D_z\} \left( \frac{d\Prob}{d \bar{ \bQ }^n} \right)^2 \right] = E _{\Prob} \biggl[ I\{ X ^n (T) \in D_z\} \frac{d\Prob}{d \bar{ \bQ }^n} \biggr].
\end{align*}}
\subsection{Importance sampling for the process $X^n$}
\label{sec:ISmodel}
\PN{To facilitate importance sampling, the dynamics of the process $X^n$ can be described by characterising the $Q^n$ process as follows (recall that $X^n$ is a scaled version of $Q^n$): The jump intensity for $Q^n$ for going from a state $n\bx = n(x_1, \dots, x_d)$ to $n\bx+\be _{j} = (nx_1, nx_2, \dots, nx_j +1, \dots, nx_d)$ is
\begin{align*}
	r^{n}(\bx; \be_{j}) = n\lambda_{j}(\bx). 
\end{align*}
Then the total jump intensity, when in state $n\bx$, is given by
\begin{align*}
R(\bx) = \sum _{j=1} ^d r^n (\bx;\be_j) = \sum _{j=1} ^d n \lambda _j(\bx).
\end{align*}
Let $T_{1}, T_{2}, \dots$ be the jump times of $Q^{n}$, $T_0 = 0$, and $\tau _k = T_k - T_{k-1}$ the time between jumps, $k=1,2,\dots$. The stochastic kernel of $Q^{n}$ is then given by
\begin{align}
\begin{split}
\label{eq:defKernel}
	\Theta ^n (dt, \be_{j} \mid \bx) &= \Prob ( \tau _{k+1} = dt, Q^{n}(T_{k+1}) - Q^{n}(T_{k}) = \be_{j} \mid Q^{n}(T_{k}) = n\bx )
	\\ &= r^{n}(\bx; \be_{j})e^{-R(\bx)t}dt,
	\end{split}
\end{align}
where $k$ is some integer.

The dynamics of the process $X^n$ are determined by the stochastic kernel $\Theta ^n$ in \eqref{eq:defKernel}.} Importance sampling amounts to using a different stochastic kernel $\bar \Theta ^n$,
\begin{align}
\label{eq:ISkernel}
	\bar{\Theta}^{n}(dt, \be_{j} \mid x) &= \bar r^{n}(\bx; \be_{j})e^{-\bar R(\bx)t}dt, 
\end{align}
with $\bar R(\bx) = \sum_{j=1}^{d} \bar r^{n}(\bx; \be_{j})$ and jump intensities $\bar r ^n (\bx,\cdot)$ of the form
$$ \bar r ^n (\bx, \be_j) = n \bar \lambda _j(\bx),$$
for some vector $\bar \lambda (\bx) = (\bar \lambda_1 (\bx),...,\bar \lambda_d(\bx))$. That is, similar to $\lambda$ and $r^n$, $\bar \lambda$ is a function from $D$ to $[0,\infty ) ^d$ and the jump intensities $\bar r ^n(\bx,\cdot)$ are obtained by scaling this function by $n$. Hence, the choice of stochastic kernel $\bar \Theta ^n$ is determined by the choice of $\bar \lambda$.

For $z \in (0,1)$, define $N^z$ as the number of jumps required for the process to reach the target set $D_z$,
\begin{align*}
N^{z} &= \inf\left\{k \geq 1: \sum_{j=1}^{d} Q^{n}_{j}(T_{k}) \geq nz \right\} = \inf \Big \{k \geq 1:X^n (T_k) \in D_z \Big \},
\end{align*}
and $N^0$ as the number of jumps needed to exceed time $T$,
\begin{align*}
N^{0} & = \inf\{k \geq 1: T_{k} > T\}.
\end{align*}
 A single sample of the importance sampling estimator based on $\bar \Theta ^n$ \PN{in \eqref{eq:ISkernel}} is given by
\begin{align}
\label{eq:estimator}
	\widehat{p}_{n} = I\{N^{z} < N^{0} \} \prod_{k=1}^{N^{z}} \frac{\Theta ^n (
	\tau_{k},v_{k}\mid X ^n (T_{k-1}))}{\bar \Theta ^n (\tau_{k},v_{k}\mid X^n (T_{k-1}))},
\end{align}
where $v_{k} \in \{\be_{1}, \dots, \be_{d}\}$ is the direction of the $k$th jump and the $\tau_k$s denote times between jumps (see Section \ref{sec:Model}). \PN{The goal is now to choose the stochastic kernel $\bar{\Theta} ^n$ so that the second moment $E_{\bar{\bQ} ^n} [\widehat p_n ^2]$ is as small as possible.}


\subsection{Hamilton-Jacobi equation and choice of sampling distribution}
\label{sec:HJperf}
\PN{Starting from a large deviation scaling of the second moment $E_{\bar{\bQ} ^n} [\widehat p_n ^2]$, and defining a related stochastic control problem, Dupuis, Wang and co-authors, have established that efficient importance sampling algorithms are connected to certain PDEs of Hamilton-Jacobi type, see for example \cite{DupuisWangSubsol, BD19}. More precisely, the form of the PDE follows from minimising a particular value function associated with the second moment $E _{\bar{\bQ} ^n} \left[ \widehat p _n ^2 \right]$ and an asymptotic analysis reminiscent of the weak convergence approach to large deviations \cite{Dupuis97}. In the setting considered in this paper, the PDE, henceforth referred to as the Isaacs equation, of interest is defined as follows: A function $W: [0,T] \times \bR ^d$, \rev{with $W_t$ denoting the time-derivative of $W$,} solves the Isaacs equation if it satisfies, in an appropriate sense,
 \begin{align}
  \label{eq:IsaacsFinal}
  \begin{cases}
   	W_t (t,\bx) - 2 H \left(x, - \frac{1}{2}\gradx W(t,\bx) \right) = 0 & (t,\bx) \in [0,T) \times D \setminus D_z, \\
	W(T,\bx) = 0, & \bx \in D_z,
	\end{cases}
 \end{align}
and $W(T,\bx) = \infty$ for $\bx \notin D_z$. For importance sampling, the sampling distribution should be constructed from a subsolution to \eqref{eq:IsaacsFinal}; a formal derivation of this equation is provided in Appendix \ref{sec:Isaacs}.}

A classical subsolution to \eqref{eq:IsaacsFinal} is a continuously differentiable function $\bar W$ that satisfies 
\begin{align}
\label{eq:SubsolCond1}
	\bar W_t (t,\bx) -2H\biggl(\bx,- \frac{1}{2}\gradx W(t,\bx)\biggr) \geq 0, \ (t,\bx) \in [0,T) \times (D \setminus D_z),
\end{align}
and
\begin{align}
\label{eq:SubsolCond2}
	\bar W(T,\bx) \leq 0, \ \bx \in D_z.
\end{align}
Suppose that $\bar W$ is such a subsolution to \eqref{eq:IsaacsFinal}. The formal derivation of this Hamilton-Jacobi equation, specifically Proposition \ref{prop:Hamiltonian}, suggests that the sampling distribution $\bar \bQ ^n$ should be constructed from $\bar W$ by using jump intensities
\begin{align}
\label{eq:COM}
	\bar \lambda _ j (\bx) = \lambda _j(\bx) \textrm{exp} \left\{- \frac{1}{2}\langle \gradx \bar W(t,\bx), \be_j \rangle \right \}, \ j=1,...,d.
\end{align}
This is the form to be used for sampling distributions throughout the remainder of this paper. \PN{Starting with \cite{DupuisWangSubsol}, it has been shown for a range of models—in \cite{DupuisWangSubsol} the authors consider primarily sums of iid random variables and empirical measures of finite-state Markov chains—the performance of an importance sampling algorithm based on a subsolution $\bar W$ is determined by the initial value $\bar W (0,0)$. In Theorem \ref{thm:asymptOpt} we prove this result for the model under consideration here. First, a detour on  the connection between the variational problem $U$ defined in \eqref{eq:LDvar} and subsolutions to \eqref{eq:IsaacsFinal}.}

Recall that $U(t,\bx)$ is the variational representation of the large deviation rate of the probability of reaching the set $D_z $ before time $T$, starting in $\bx$ at time $t$. It turns out that the function $U(t,\bx)$ is a viscosity solution to the Hamilton-Jacobi equation
\begin{align}
\label{eq:LDPDE}
\begin{cases}
	U_t(t,\bx) - H(x,-\gradx U(t,\bx)) = 0, & (t,\bx) \in [0,T) \times (D \setminus D_z), \\
	U(T,\bx) = 0, & \bx \in D_z.
\end{cases}
\end{align}
\PN{For a range of different assumptions on the Hamiltonian $H$, this result has been established at varying levels of rigour, see for example \cite{Evans10}. For assumptions suitable for this paper, a rigorous proof} is provided in \cite{DHN_2013}. Moreover, subsolutions to the equation \eqref{eq:LDPDE} give rise to subsolutions to the equation \eqref{eq:IsaacsFinal}. Indeed, a subsolution $\bar U$ to \eqref{eq:LDPDE} satisfies
\begin{align*}
	2 \bar U _t  (t,\bx) - 2 H \left( x, -\frac{2 \gradx \bar  U (t,\bx)}{2} \right) = 2 \left( \bar U _t (t,\bx) - H (\bx, -\gradx \bar U (t,\bx))\right) \geq 0,
\end{align*}
where the inequality follows from the subsolution property. It follows that $2 \bar U$ is a subsolution to the Hamilton-Jacobi  equation \eqref{eq:IsaacsFinal}. Therefore, to construct efficient sampling algorithms it suffices to consider subsolutions to the Hamilton-Jacobi equation \eqref{eq:LDPDE}. This also implies that if the function $U$ can be computed explicitly, asymptotically optimal importance sampling algorithms can be obtained by using $\bar W = 2U$ in \eqref{eq:COM}.

We are now ready \PN{to state and prove the main result regarding performance of importance sampling algorithms based on subsolutions to \eqref{eq:IsaacsFinal}, for the model under consideration.}
\begin{theorem}
\label{thm:asymptOpt}
Let $\bar W$ be a subsolution of \eqref{eq:IsaacsFinal} which is $C^1(D)$. 
If $\widehat p _n$ is the importance sampling estimator based on the vector of jump intensities $\bar \lambda$ defined in \eqref{eq:COM}, then
\begin{align*}
	\limsup _{n \to \infty} \frac{1}{n} \log E _{\bar \Theta ^n} \left[ \hat p _n ^2\right] \leq -\frac{\bar W(0,0)}{2} - U(0,0).
\end{align*}
\end{theorem}
Before embarking on the proof of Theorem \ref{thm:asymptOpt}, we note that the formal derivation of the Isaacs equation \eqref{eq:IsaacsFinal} in Appendix \ref{sec:Isaacs} suggests that the performance, as measured by the relative error, is determined by the initial value of $\bar W$. This is precisely the conclusion of Theorem \ref{thm:asymptOpt} and thus there is no need to make the formal arguments of the derivation rigorous; the derivation serves to provide intuition for the form of the Hamilton-Jacobi equation.


\begin{proof}
The likelihood ratio between the sampling distribution $\bar{\bQ} ^n$ (stochastic kernel $\bar \Theta ^n$) and the original distribution $\P$ (stochastic kernel $\Theta ^n$) can be expressed as
	\begin{align*}
		\frac{d \P}{d \bar{\bQ} ^n} = \textrm{exp} \Biggl \{ \int _{0} ^{T _{N^z}} (\bar R (X^n(s)) - R(X^n (s)))ds + \sum _{k=1} ^{N^z} \log  \frac{r ^n (X ^n (T_{k-1}), v_k)}{\bar r ^n (X ^n (T_{k-1}), v_k)}  \Biggr\}. 
	\end{align*}
	To analyze the expectation of the likelihood ratio, define the measure $m^n (\cdot,\cdot)$ on $\R ^d$, given $\R ^d$, by
	\begin{align*}
		m^n(\bx,d\by) = \sum _{j=1} ^d n \lambda _j(\bx)\delta _{\be_j}(d\by), \ \bx\in \R ^d .
	\end{align*}
	Furthermore, let $\{ M^n (t,\cdot) \}_t$ denote the point process defined by the jumps of $X^n$. That is, $M^n(t,B)$ is the number of jumps of $X^n$ in $(0,t]$ in directions that are in $B \subset \R ^d$,
	\begin{align*}
	M^n (t, B) = n \sum _{j =1} ^d X^n _j(t) I\{\be_j \in B \}.
	\end{align*} 
	At any time $t$, the instantaneous jump intensity of $M^n$ is $m^n(X^n(t),\cdot)$.
	
	With the jump intensities $\bar \lambda$ taken as in \eqref{eq:COM}, the likelihood ratio can be expressed as
	\begin{align*}
		\frac{d \P}{d \bar{\bQ} ^n} = \textrm{exp} &\left \{ \int _0 ^{T_{N^z}} n H\left (X^n(s), - \frac{1}{2}\gradx \bar W (s, X^n(s)) \right )ds \right. \\
		& \left. \quad + \frac{1}{2} \sum _{k=1} ^{N^z} \langle \gradx \bar W (T_{k-1}, X^n(T_{k-1})), v_k\rangle \right\}.
	\end{align*}
Using the definitions of $m^n$ and $M^n$, the likelihood ratio takes the form
	\begin{align*}
		\frac{d \P}{d \bar{\bQ} ^n} = \textrm{exp} & \left \{ \int _0 ^{T_{N^z}} n H\left (X^n(s), - \frac{1}{2}\gradx \bar W (s, X^n(s)) \right )ds \right. \\
		& \quad + \frac{n}{2} \int _0 ^{T_{N^z}} \int _{\R ^d} \left( \bar W (s, X^n(s) + \frac{\by}{n} ) - \bar W (s, X^n(s)) \right) dM^n(s,\by)  \\
		& \quad - \frac{n}{2} \sum_{k=1} ^{N^z} \left( \bar W (T_{k-1}, X^n(T_{k-1}) + v_k /n ) - \bar W (T_{k-1}, X^n(T_{k-1})) \right) \\
		& \quad \left.+ \frac{n}{2} \sum _{k=1} ^{N^z} \langle \gradx \bar W (T_{k-1}, X^n(T_{k-1})), v_k /n \rangle \right\} .
	\end{align*}		
By partial integration 

	\begin{align*}
		& \int _0 ^{T_{N^z}} \int _{\R ^d} \left( \bar W (s, X^n(s) + \by/n ) - \bar W (s, X^n(s)) \right) dM^n(s,\by)\\
		& \quad = \bar W (T_{N^z}, X^n(T_{N^z})) - \bar W (0,0) -  \int _0 ^{T_{N^z}} \bar W _t (s, X^n (s)) ds. 
	\end{align*}
Since $\bar W$ is assumed to be $C^1(D)$ and the state space $D$ is a compact subset of $\R^d$, the convergence
\begin{align*}
	n \left( \bar W \left(t, \bx + \frac{\be_j}{n}\right) - \bar W(t,\bx) \right) \to \langle \gradx \bar W(t,\bx), \be_j \rangle,
\end{align*}
as $n \to \infty$, is uniform in $\bx$. Hence, there is a sequence $C_n$ such that $C_n \to 0$ as $n \to \infty$, and
\begin{align*}
	\sup _{\bx \in D, j\in\{1,\dots ,d\}} \left| \langle \gradx \bar W(t,\bx), \be_j \rangle - n \left( \bar W \left(t, \bx + \frac{\be_j}{n}\right) - \bar W(t,\bx) \right) \right| \leq C_n.
\end{align*}
The uniform convergence thus implies the upper bound
\begin{align*}
	N^z C_n \geq n \sum _{k=1} ^{N^z} &\left( \langle \gradx \bar W (T_{k-1}, X^n(T_{k-1})), v_k /n \rangle \right. \\
	& \quad \left. - \bar W (T_{k-1}, X^n(T_{k-1}) + v_k /n ) + \bar W (T_{k-1}, X^n(T_{k-1}))  \right),
\end{align*}
which gives an upper bound for the likelihood ratio,
\begin{align*}
	\frac{d \P}{d \bar{\bQ} ^n} \leq \textrm{exp} & \Biggl \{ - \frac{n}{2} \int _0 ^{T_{N^z}} \biggl( \bar W _t (s, X^n (s)) -2H \biggl (X^n(s), - \frac{1}{2}\gradx \bar W (s, X^n(s)) \biggr )  \biggr) ds \\
	& \quad + \frac{n}{2} \bar W (T_{N^z}, X^n(T_{N^z})) - \frac{n}{2} \bar W (0,0) + \frac{1}{2} N^z C_n \Biggr \}.
\end{align*}		
The assumption that $\bar W$ is a subsolution to \eqref{eq:IsaacsFinal} implies that the first integral is bounded from below by 0. Moreover, by the definition of $N^z$, $\bar W (T_{N^z}, X^n(T_{N^z})) \leq 0$. Hence, the following upper bound holds for $\widehat p _n = I\{ N^z < N^0 \} (d \P / d \bar{\bQ} ^n)$,
\begin{align*}
E_{\Theta ^n} \biggl[ I\{N^z < N^0 \} \frac{d \P}{d \bar{\bQ} ^n} \biggr] &\leq E_{\Theta ^n} \biggl[ I\{ N^z < N^0 \} \textrm{exp} \biggl\{ - \frac{n}{2} \bar W (0,0) + \frac{1}{2} N^z C_n  \biggr\} \biggr] \\
&= e^{- \frac{n}{2}\bar W(0,0)}E _{\Theta ^n} \left[ I\{ N^z < N^0\} e^{\frac{1}{2} N^z C_n} \right].
\end{align*}
\PN{Due to the form of the jump intensity $r^n$, the process $Q^n$, hence also the process $X^n$, can have at most $n$ jumps. It follows that $N^z \leq n$}. Combined with the upper bound just derived for the expectation of $\widehat p_n$, this yields the upper bound
\begin{align*}
	\frac{1}{n} \log E_{\Theta ^n} \left[ I\{N^z < N^0 \} \frac{d \P}{d \bar{\bQ} ^n} \right] & \leq - \frac{1}{2} \bar W(0,0) + \frac{C_n}{2} + \frac{1}{n}\log p_n.
\end{align*} 		
The result now follows from the large deviation principle for $p_n$ and the \PN{properties of the sequence $C_n$},
\begin{align*}
	\limsup _{n \to \infty} 	\frac{1}{n} \log E_{\Theta ^n} \left[ I\{N^z < N^0 \} \frac{d \P}{d \bar{\bQ} ^n} \right] \leq -\frac{1}{2} W(0,0) - U(0,0).
\end{align*}
\end{proof}
\section{Subsolutions and associated sampling algorithms}
\label{sec:Subsol}
In this section we construct a type of subsolution to \eqref{eq:IsaacsFinal} that gives rise to efficient importance sampling algorithms for the model described in Section \ref{sec:Model}. Recall from Section \ref{sec:IS} that the Isaacs equation \eqref{eq:IsaacsFinal} of interest is
\begin{align*}
	 \begin{cases}
   	W_t (t,\bx) - 2 H \left(x, - \frac{1}{2}\gradx W(t,\bx) \right) = 0 & (t,\bx) \in [0,T) \times D \setminus D_z, \\
	W(T,\bx) = 0, & \bx \in D_z,
	\end{cases}
\end{align*}
and for a subsolution $\bar W$ of this equation the corresponding importance sampling algorithm is defined by using jump intensities \eqref{eq:COM}:
\begin{align*}
	\bar \lambda _ j (\bx) = \lambda _j(\bx) \textrm{exp} \left\{- \frac{1}{2}\langle \gradx \bar W(t,\bx), \be_j \rangle\right \}, \ j=1,...,d.
\end{align*}
Before considering arbitrary dimension $d$, we begin with finding the optimal time-homogeneous sampling distribution for $d=1$. The one-dimensional case sets the stage for the more general construction and provides insight into the underlying idea. 
\subsection{The optimal time-homogeneous sampling distribution}
\label{sec:timeHom}
As an illustration of the idea for how to construct subsolutions for arbitrary dimension $d$, a special case of a method described in \cite{DHN_2013}, we start by considering the case $d=1$. \PN{Recall that to evaluate the performance of any importance sampler based on a subsolution $\bar W$, the initial value $\bar W(0,0)$ should be compared to $2U(0,0)$, with $U$ as in \eqref{eq:LDvar}. To facilitate such a comparison, we have the following result.}
\begin{proposition}
\label{prop:asympOpt1}
	For $d=1$, the large deviation rate is given by
	\begin{align*}
		U(0,0) = \int _0 ^z \log \left( 1 + \frac{ c }{\lambda(y)} \right) dy - c T,
	\end{align*}
where $c$ solves the equation
\begin{align*}
	\int _0 ^z \frac{dy}{\lambda(y) + c} = T.
\end{align*}
\end{proposition}
\PN{This result can be proved using e.g.\ convex optimisation arguments. However, it is a special case of a more general result by the authors (Theorem \ref{thm:subsol}), presented later in this section and a separate proof is omitted.}

As mentioned in the introduction we now only consider algorithms for which the change of measure is independent of $t$. That is, if $\bar W$ is the subsolution from which the sampling distribution is constructed, then $\nabla _x \bar W(t,x)$ is a function of only \PN{$x \in \bR$; even though $x\in \bR$ here, we use the notation $\nabla _x \bar W$ for the derivative with respect to $x$ to facilitate comparison with the rest of the paper. To emphasize the assumption that $\nabla _x \bar W(t,x)$ is a function of only $x$, not $t$,} and to ease notation, let $\alpha (x) = - \nabla _x \bar W(t,x) / 2$. Then the jump intensity used for importance sampling can be expressed as 
\begin{align*}
	\bar \lambda (x) = \lambda (x) \textrm{exp}\left\{-\frac{1}{2}\nabla _x \bar W(t,x)\right\} = \lambda(x) \textrm{exp}\left\{\alpha(x)\right\}.
\end{align*}
\PN{From the assumptions on $\bar W$, we can represent such a function as}
\begin{align*}
	\bar W(t,x) = -2 \int _0 ^x \alpha(y) dy + g(t) + K,
\end{align*}
for some function $g$ and constant $K$. Let $A (x) = \int _0 ^x \alpha(y) dy$ and consider only functions $g$ of the form $g(t) = 2 c t$, for some constant $c$. Then,
\begin{align*}
	\bar W_{t}(t,x) = 2c, \quad \nabla _x \bar W(t,x) = -2\alpha(x).	
\end{align*}
For $\bar W$ to be a subsolution, $A$, $g$ and $K$ must be chosen so that conditions \eqref{eq:SubsolCond1}-\eqref{eq:SubsolCond2} are satisfied. For \eqref{eq:SubsolCond1} to hold, $\alpha$ must be such that
\begin{align*}
	0 &\leq \bar W_t (t,x) - 2 H \left(x, - \frac{1}{2}\nabla _x \bar W(t,x) \right) = 2c - 2 \lambda(x) \left( e^{\alpha (x)} - 1 \right),
\end{align*}
which implies that $\alpha(x)$ must satisfy
\begin{align*}
	\alpha(x) \leq \log \left( 1 + \frac{c}{\lambda(x)} \right).
\end{align*}
By setting $\alpha(x)$ equal to the right-hand side of the previous display, equality is achieved in \eqref{eq:SubsolCond1}. However $\alpha$ is well-defined only for $c \geq - \inf _{x \leq z} \lambda(z)$.  For this particular choice of $\bar W$ the terminal condition \eqref{eq:SubsolCond2} becomes
\begin{align*}
	0 & \geq \bar W(T,z) = 2c T - 2A(z) + K,
\end{align*}
and the constant $K$ must satisfy
\begin{align*}
	K \leq 2 A(z) - 2 c T.	
\end{align*}
From the discussion in Section \ref{sec:IS} and at the beginning of this section, it is clear that it is desirable to have the initial value $\bar W(0,0)$ as large as possible. Here, $\bar W(0,0) = K$ and the inequality gives the upper bound $2A(z)- 2cT$; take $K$ to be equal to this upper bound. The resulting subsolution $\bar W$ is given by
\begin{align*}
	\bar W(t,x) = 2 \int _x ^z \log \biggl( 1 + \frac{c}{\lambda (y)} \biggr)dy - 2c (T-t).
\end{align*}
Lastly, the constant $c$ can now be chosen so as to maximize $\bar W(0,0)$: Take $c = c ^*$,
\begin{align*}
	c^{*} = \textrm{argmax } \bar W (0,0) = \textrm{argmax }  \left\{ 2 \int_{0}^{z} \log\left(1+\frac{c}{\lambda(y)}\right) dy - 2 c T \right\},
\end{align*}
where only $c > - \inf _{x \leq z} \lambda(x) $ are considered. Differentiability with respect to $c$ implies that the optimal $c ^*$ must be a solution to the equation
\begin{align}
\label{eq:OptimalC}
	\int _0 ^z \frac{dy}{\lambda(y) + c} = T.
\end{align}
For this choice of $c^{*}$ the subsolution $\bar W$ has initial value
\begin{align*}
	\bar W(0,0) = 2 \int_{0}^{z} \log\left(1+\frac{c^{*}}{\lambda(y)}\right) dy - 2c^{*} T. 
\end{align*}
\PN{Comparing this expression to that of Proposition \ref{prop:asympOpt1}, we see that the importance sampling algorithm associated with $\bar W$ is asymptotically optimal.} 

To present the more general result, we begin with \PN{introducing the $\Mane$ potential \cite{Mane97}}. For $c \in \R$ and $\bx, \by \in \R ^d$, the \emph{$\Mane$ potential at level c}, denoted by $S^c(\bx,\by)$, is defined as
\begin{align*}
	S^c (\bx,\by) = \inf \left\{ \int _0 ^{\tau} \left( c + L(\psi (s), \psi' (s)) \right )ds, \ \psi(0) = \bx, \ \psi(\tau) = \by \right\},
\end{align*}
where $L$ is the local rate function defined in Section \ref{sec:LD} and the infimum is taken over $\psi \in \ac ([0,\infty) \colon \R ^d)$ and $\tau > 0$. In the current setting, because $X^{n}(0) = 0$ (\PN{in the context of credit risk, this corresponds to no defaults at time $0$}), the initial value of interest is $\bx _0=0$. The following \PN{results, Lemma \ref{lemma:Mane} and Theorem \ref{thm:subsol}, show how one can construct subsolutions, and thus importance sampling algorithms, using the relation between the Ma\~n\'e potential $S^c(0,\by)$ and the Hamiltonian $H$. The first result establishes that $S^c(0,\by)$ is a subsolution, a type of generalised solution often used for Hamilton-Jacobi equations \cite{Evans10}, to a stationary equation involving the Hamiltonian $H$}.
\begin{lemma}[cf. Proposition 2.2 in \cite{DHN_2013}]
\label{lemma:Mane}
	The function $\by \mapsto S^c(0,\by)$ is a viscosity solution of the stationary Hamilton-Jacobi equation
	\begin{align}
	\label{eq:stationary}
		H(\by, \nabla  S(\by)) = c,
	\end{align}
	for all $\by \neq 0$ and $c > c_H$, where $c_H$ satisfies
	\begin{align*}
	c_H \geq \sup _{\bx} \inf _{\mathbf{p}} H(\bx,\mathbf{p}).
	\end{align*}
\end{lemma}
The constant $c_H$ is known as \emph{$\Mane$'s critical value}, \PN{see \cite{Mane97}}. For the specific model considered here (see Example 2.1 in \cite{DHN_2013}),
\begin{align*}
	c_H = - \inf _{\bx \in D \setminus  D_z} \sum _{j=1} ^d \lambda _j(\bx).
\end{align*}
Note that this is consistent with the discussion leading up to Proposition \ref{prop:asympOpt1}, where it was necessary to take $c > - \inf _{x \leq z} \lambda(x)$ \PN{(recall that Proposition 5.1 concerns the one-dimensional setting, $x,z \in \bR$)}. 

The following theorem is a combination of results in \cite{DHN_2013} adapted to the current setting. \PN{It is a generalisation to higher dimensions of the method used for the one-dimensional case; for $d=1$ the theorem establishes the optimality of the importance sampling algorithm previously derived.}

\begin{theorem}[cf.\ Sections 3 and 5 in \cite{DHN_2013}]
\label{thm:subsol}
\PN{Suppose $S : \bR ^d \to \bR$ is a classical subsolution to \eqref{eq:stationary}.
Consider the collection of functions $\bar U ^{c,\by} : \bR \times \bR ^d \to \bR$ defined by
	\begin{align*}
		\bar U^{c, \by} (t, \bx) = S(\by) - S(\bx) - c(T-t)
	\end{align*}
where $c>c_H$ and $\by \in \partial D_z$. Then $\bar U^{c,\by} (\cdot, \cdot)$ is a classical subsolution to \eqref{eq:LDPDE}. Specifically, if $\by \mapsto S^c (0,\by)$ is $C^1(D)$, then the corresponding $\bar U ^{c, \by}$ is a classical subsolution to \eqref{eq:LDPDE}.}

Moreover, for $d=1$, if we take $y=z$ and $c$ such that 
\begin{align*}
	\bar U ^{c,z}(0,0) = \sup _{c > c_H} \{ S^c(0,z) - cT \},
\end{align*}
the corresponding $\bar U ^{c,z} (0,0)$ is equal to value of the large deviation rate function at $(0,0)$;
\begin{align*}
	\bar U ^{c,z} (0,0) = U(0,0).
\end{align*}	
\end{theorem}
The result together with the preceding discussion states that, for $c>c_H$ and $y \in \partial D _z$, $\bar W = 2 \bar U ^{c,\by}$ is a subsolution to the equation \eqref{eq:IsaacsFinal}. Moreover, in the case $d=1$, it states that $\bar W$ attains the maximal initial value $2U(0,0)$ if we choose $c$ appropriately. It follows that the corresponding choice of the sampling distribution achieves asymptotic optimality; see Section \ref{sec:IS} for a rigorous proof. 

It is important to point out that Theorem \ref{thm:subsol} and the results on performance (primarily Theorem \ref{thm:asymptOpt}) do not depend on the explicit form of the original jump intensity, described by  $\lambda$. Rather, both hold for any Markovian birth process with the structure described in Section \ref{sec:Model}. However, the particular choice of $\lambda$ becomes crucial when computing the explicit change of measure for a specific model, which in the context of Theorem \ref{thm:subsol} amounts to computing the $\Mane$ potential.

For $d=1$, $\partial D _z = \{ z \}$ and the $\Mane$ potential is precisely the function
\begin{align*} 
	S^c(0,x) = \int _0 ^x \log \biggl( 1 + \frac{c}{\lambda(y)} \biggr)dy.
\end{align*}
Thus, the construction of $\bar W$ according to Theorem \ref{thm:subsol} is precisely the subsolution constructed in Section \ref{sec:timeHom}, and \PN{Proposition \ref{prop:asympOpt1} becomes a corollary of Theorem \ref{thm:subsol}; asymptotic optimality follows from a combination of Proposition \ref{prop:asympOpt1} and Theorem \ref{thm:asymptOpt}.}


\subsection{Sampling distribution for multi-dimensional model}
\label{sec:choiceDist}
\PN{We now approach the task of finding an explicit change of measure for the model \eqref{eq:Model}, where the jump intensities are of the form
\begin{align*}
	\lambda _j (\bx) = a_j (w_j-x_j) e^{b \sum _{i=1} ^d x_i}, \ j=1, \dots, d,
\end{align*}
for some non-negative $a_j$s and $b$. Recall (see Section \ref{sec:Model}) that the original model in \cite{CarmonaCrepey} did not allow for different jump intensities for the different components of the process $Q^n$, i.e.\ they only considered the case $a_i \equiv a$, $i=1, \dots , d$, for some intensity $a$, rendering the model essentially one-dimensional. The model \eqref{eq:Model}, and by extension the simulation algorithms constructed in this paper, is thus a generalisation of that considered in \cite{CarmonaCrepey}.}
%

\PN{As outlined in the previous sections, the idea is to construct a suitable subsolution: grounded in Theorem \ref{thm:subsol}, the idea is to find $\gradx \bar W (t,\bx) = - 2 \gradx S^{c}(0,\bx)$.} For $d=1$, the derivation of $S^c(0,y)$ is provided in Section \ref{sec:timeHom} and the corresponding sampling distribution has jump intensity $\bar \lambda(x) = \lambda(x) e^{\alpha (x;c)}$, where
\begin{align*}
	\alpha (x, c) = \log \biggl( 1 + \frac{c}{\lambda(x)} \biggr),
\end{align*}
and $c$ solves \eqref{eq:OptimalC}. Hence, for the one-dimensional model the change of measure used for importance sampling is completely known up to the constant $c$, which one might need to determine numerically. 

\PN{For $d\geq 2$, a natural approach is to try and solve the variational problem in the definition of $S^{c}(0,\bx)$. Another approach is the one (implicitly) used for $d=1$: Find a function $\alpha (\bx;c)$ that solves the equation 
\begin{align}
\label{eq:alphaH}
	c = H(\bx,\alpha(\bx,c)) = \sum _{j=1} ^d \lambda _j(\bx) \left( e^{\langle \alpha (\bx,c), \be_j \rangle } -1 \right),
\end{align}
and then attempt to find a potential $A(\bx;c)$ such that $\gradx A(\bx;c) = \alpha(\bx;c)$.
However, for $d>1$, this requires that $\alpha (\bx ;c)$ forms a conservative vector field. Finding such solutions to \eqref{eq:alphaH} clearly depends on the choice of $\lambda$ and is a non-trivial task already for rather simple choices. }

Before discussing further the problem of finding efficient sampling distributions for $d>1$, \PN{we return to the special case considered in \cite{CarmonaCrepey}: $a_j \equiv a$, $j=1,\dots,d$, so that all groups are homogeneous. In this case the change of measure can once again be found explicitly by choosing $\alpha(\bx;c)$ according to}
\begin{align}
\label{eq:alpha}
	\langle \alpha(\bx;c), \be_j\rangle = \log \left( 1 + \frac{c}{\sum _{i=1} ^d \lambda _i (\bx)} \right), \ j=1, \dots, d.
\end{align}
This defines a conservative vector field and the corresponding potential $A(\bx;c)$, as well as the optimal $c$, is analogous to before, with
\begin{align*}
	A(\bx;c) = \int _0 ^{\sum x_i} \log \left( 1 + \frac{c}{\tilde \lambda(y)} \right)dy,
\end{align*}
\PN{where $\tilde \lambda (y) = a(1-y) e^{b\by}$.} Note that this relies on the form of \PN{the $\lambda_j$s}, specifically the fact that $\sum _i \lambda _i (\bx)$ is a function of $\bx=(x_1,\dots, x_d)$ only through the sum $\sum _i x_i$. An interesting observation is that this choice of sampling distribution amounts to $\gradx \bar W$ being perpendicular to the barrier the process is trying to cross, which seems intuitively appealing.

The function $\alpha (\bx,c)$ defined by \eqref{eq:alpha} is a solution to the stationary equation \eqref{eq:alphaH} \PN{for the general model \eqref{eq:Model}, with different intensities $a_i$,} not just the effectively one-dimensional case of \cite{CarmonaCrepey}. \PN{It is therefore tempting to base the sampling distribution on this choice also for the general model}. However, this $\alpha(\bx,c)$ is not necessarily a conservative vector field: for $d=2$ the (scalar) curl of $\alpha(\bx,c)$ is such that the necessary condition for existence of a potential $A(\bx,c)$ with $\alpha (\bx,c) = \gradx A(\bx,c)$ becomes
\begin{align*}
	\frac{e^{b(x_1 + x_2)}}{\left( \lambda_1(\bx) + \lambda _2(\bx) \right)^2 + \lambda_1(\bx) + \lambda _2(\bx)} \left(a_1 - a_2\right) = 0,
\end{align*}
which does not hold when the two groups have different intensities. At the moment, for general choices of $\mathbf{a} \in \R ^d$ and $b > 0$, it is not known to the authors how to find the $\Mane$ potential $S^{c}(0,\bx)$ corresponding to \eqref{eq:Model}; \PN{for example, for the case $b=0$ the above choice does produce a conservative vector field and we can use the corresponding subsolution.}  
\PN{Still, the obtained results can be used to guide the design of sampling algorithms and next we discuss one suggestion.} 

Consider the general form of the jump intensity model, where the $a_j$'s are non-negative and not necessarily equal. Let $a^* = \vee _{j=1} ^d a_j $ and define the vector $\alpha (\bx; c)$ by
\begin{align*}
	\langle \alpha  (\bx;c) ,\be_j \rangle = \log \left( 1 + \frac{c}{\sum _{j=1}^d a^* (w_j-x_j) e^{b \sum x_i}} \right), \ j=1, \dots, d.
\end{align*}
This choice of $\alpha (\bx;c)$ satisfies
\begin{align*}
	H(x,\alpha (\bx;c)) = c \left( \frac{\sum _{j=1} ^d a_j (w_j -x_j)}{ \sum _{j=1}^d a^* (w_j-x_j) }  \right) \leq c,
\end{align*}
and thus
\begin{align*}
	c - H(\bx, \alpha (\bx;c)) \geq 0.
\end{align*}
This is the correct inequality for a subsolution to the stationary Hamilton-Jacobi equation. Let $A (\bx; c)$ be the potential for the vector field $\alpha (\bx;c)$ and define the corresponding $\bar W (t,\bx)$ by
\begin{align*}
	\bar W (t,\bx) = 2 A (\by;c) - 2A(\bx;c) - 2c (T-t).
\end{align*}
This is a subsolution to \eqref{eq:IsaacsFinal} and we base our importance sampling algorithm on this choice of $\bar W$. for the special case of \cite{CarmonaCrepey}, with all groups homogeneous, $\bar W$ coincides with the optimal subsolution. However, optimal performance is no longer guaranteed by Theorem \ref{thm:subsol}.; 


\PN{ In order to use this choice of $\bar W$ it remains to determine the constant $c$. For this, the physical interpretation of $c$ as the energy level added to the system can be used. The choice of $c$ should then be such that trajectories of $X^n$ under the sampling measure take an appropriate amount of time to reach $D_z$; finding this energy level does not add any significant computational cost}. However, good performance is no longer suggested by Theorem \ref{thm:subsol}. In lieu of theoretical results on performance, this algorithm is studied numerically in Section \ref{sec:Numerics}, \PN{for some choices of parameter values for $d=2$ and $d=3$}, exhibiting good performance in the rare-event setting.

We end the section with a brief remark on the importance sampling algorithms used in \cite{CarmonaCrepey}. Therein the jump intensities of the sampling distributions are of the form
\begin{align*}
	\bar \lambda _j(\bx) = \gamma \lambda _j (\bx),
\end{align*}
for different values of $\gamma$. This choice corresponds to a state-independent change of measure. Such jump intensities are obtained as a special case of \eqref{eq:COM} by considering subsolutions $\bar W$ that are affine in the state-variable $\bx$. However, such affine subsolutions will typically not achieve a maximal initial value, which explains the poor performance observed in \cite{CarmonaCrepey}: \PN{whereas therein the value of $\gamma$ is selected via a trial-and-error strategy, an analysis using subsolutions shows that even the optimal choice of $\gamma$ will not lead to asymptotic optimality.}


\section{Numerical experiments}
\label{sec:Numerics}
We now present some numerical experiments for the importance sampling algorithms proposed in Section \ref{sec:Subsol}, for different choices of number of groups $d$ and parameter values in \eqref{eq:Model}: \PN{ recall that the jump intensities are of the form
\begin{align*}
	\lambda _j (\bx) = a_j (w_j-x_j) e^{b \sum _{i=1} ^d x_i}, \ j=1, \dots, d,
\end{align*}
for some non-negative $a_j$s and $b$. } In particular, we implement the optimal time-homogeneous importance sampler (defined in Theorem \ref{thm:subsol}) for the \PN{one-dimensional} examples studied in \cite{CarmonaCrepey}, 
verifying numerically the asymptotic optimality \PN{that follows from a combination of Theorem \ref{thm:asymptOpt} and Proposition \ref{prop:asympOpt1}}. \PN{Note that the results presented in this section are based on the embedded discrete-time Markov chain $\{ X^n(T_j); j\geq 0\}$, as opposed to the continuous-time process $\{ X^n (t); t\in [0,T] \}$; this does not affect the performance analysis in the previous sections, which remains intact for the embedded chain, see e.g.\ Sections 7 and 9 in \cite{DupuisWangLeder_wslqIS}.}

\PN{Before reporting our results, some comments on the numerical results in \cite{CarmonaCrepey} are in place. Therein, the authors do not compute the probability of the event $\{ X^n(T) \geq z \}$, for $z \in (0,1)$, but rather of $\{ X^n (t) = i/n \}$ for each $i=1, 2, \dots , n$. Moreover, they use a range of parameters to decide the importance sampling intensity, changing the parameter for different $i \in \{ 1, 2, \dots, n \}$. Because of this, that only no or extreme contagion are considered, and the fact that no type of uncertainty for the estimates (such as relative errors) are reported, the results are not suitable for direct comparisons to those presented in this paper. The main takeaway from the numerical experiments in \cite{CarmonaCrepey} is the inefficiency of the state-independent importance sampling algorithm, and the interacting particle system method, used therein when there is moderate to extreme contagion present in the model, showing the need for state-dependent alternatives in the case of importance sampling.

In what follows our experiments are grouped into examples with homogeneous groups (i.e., one or more groups with the same intensities), which in principle can be compared to the results in \cite{CarmonaCrepey}, and examples with $d \geq2$ groups with different intensities. For all experiments, the estimates are based on $100$ batches with $N=5000$ samples in each batch; probability estimates and relative errors are computed over batches.}

\PN{
\subsubsection*{Homogeneous groups with no to moderate contagion}
For the homogeneous setting, to stay consistent with \cite{CarmonaCrepey}, the number of obligors is $n=125$, the time of maturity is $T=5$ and we have used as intensity $a=0.01$; in the credit risk context, this corresponds to a homogeneous pool of obligors, all with default rate 0.01 (when we do not account for contagion effects). In Tables \ref{table:homB0} and \ref{table:homB5} we present simulation results for the one-dimensional version of the model \eqref{eq:Model} for $b=0$ (Table \ref{table:homB0}) and $b=5$ (Table \ref{table:homB5}). The subsolution used to construct the sampling distribution is that of Theorem \ref{thm:subsol}. 

As described in Section \ref{sec:Subsol}, the homogeneous version of the model can essentially be reduced to the one-dimensional setting, regardless of the dimension $d$, and the accuracy of the asymptotically optimal simulation algorithm is not affected. To illustrate this point, we consider an example with $d=5$, $a=0.01$, $b=5$ and equal partitioning amongst the $d$ groups: $w_j = 1/d$, $j=1, \dots, d$; the numerical results are presented in Table \ref{table:homd5B5} (compare to Table \ref{table:homB5}).

\begin{table}[ht]
\caption{Importance sampling and Monte Carlo estimates for the case of independent obligors; \PN{$d=1$, $n=125$, $T=5$}, $a=0.01$, $b=0$.} \label{table:homB0}
\centering
\begin{tabular}{c | c | c | c | c}
 & \multicolumn{2}{c|}{Importance sampling} & \multicolumn{2}{c}{Monte Carlo} \\
 $z$ & $\widehat p_n$ & $\textrm{RE}(\widehat p_n)$ & $\widehat p_n$ & $\textrm{RE}(\widehat p_n)$ \\ 
 \hline
 0.10 &8.238e-3 & 0.0219 & 8.282e-3 & 0.1389 \\
 0.15 &1.089e-5 & 0.027 & 1.200e-5& 3.978 \\
 0.20 & 1.737e-9 & 0.028 & - & - \\
 0.25  & 7.250e-15 & 0.031 & - & - \\
 0.30  & 3.499e-20 & 0.039 & - & - \\
 0.35  & 4.470e-26 & 0.038 & - & - \\
 0.40  & 1.624e-32 & 0.037 & - & - \\
 \hline
\end{tabular}
\end{table}

\begin{table}[ht]
\caption{Importance sampling and Monte Carlo estimates for a model with moderate contagion; \PN{$d=1$, $n=125$, $T=5$}, $a=0.01$, $b=5$.} \label{table:homB5}
\centering
\begin{tabular}{c | c | c | c | c}
 & \multicolumn{2}{c|}{Importance sampling} & \multicolumn{2}{c}{Monte Carlo} \\
 $z$ & $\widehat p_n$ & $\textrm{RE}(\widehat p_n)$ & $\widehat p_n$ & $\textrm{RE}(\widehat p_n)$ \\ 
 \hline
 0.10 & 4.389e-2 & 0.0183 & 4.383e-2 & 0.0613 \\
 0.15 & 9.337e-4 & 0.0210 & 9.660e-4 & 0.480 \\
 0.20 & 9.183e-6 & 0.0266 & 1.800e-5 & 3.196 \\
 0.25 & 2.552e-8 & 0.0273 & - & - \\
 0.30 & 1.380e-10 & 0.0295 & - & - \\
 0.35 & 7.280e-13 & 0.0343 & - & - \\
 0.40  & 4.089e-15 & 0.0322 & - & - \\
 \hline
\end{tabular}
\end{table}

\begin{table}[ht]
\PN{
\caption{\PN{Importance sampling and Monte Carlo estimates for a homogeneous model with moderate contagion; $d=5$, $n=125$, $T=5$ , $a_j=0.01$, $j=1, \dots, 5$, $b=5$, $w_j = 1/5$, $j=1, \dots, 5$.}} \label{table:homd5B5}
\centering
\begin{tabular}{c | c | c | c | c}
 & \multicolumn{2}{c|}{Importance sampling} & \multicolumn{2}{c}{Monte Carlo} \\
 $z$ & $\widehat p_n$ & $\textrm{RE}(\widehat p_n)$ & $\widehat p_n$ & $\textrm{RE}(\widehat p_n)$ \\ 
 \hline
 0.10 & 4.391e-2  & 0.0202 & 7.800e-2 & 0.0632 \\
 0.15 & 9.334e-4 & 0.0292 & 9.200e-4 & 0.516 \\
 0.20 & 8.447e-6 & 0.0362 & 2.000e-6 & 10.000 \\
 0.25 & 2.565e-8 & 0.0454 & - & - \\
 0.30 & 1.385e-10 & 0.0543 & -  & - \\
 0.35 & 7.247e-13 & 0.0577 & - & - \\
 0.40 & 4.102e-15 & 0.0734 & - & - \\
 \hline
\end{tabular}}
\end{table}
The results in Tables \ref{table:homB0}–\ref{table:homd5B5} illustrate the asymptotically optimal performace of the proposed method in the case of a homogeneous model (i.e., all intensities $a_j$ are equal). Table \ref{table:homB0} corresponds to the example of independent obligors ($b=0$) studied in Section 4 in \cite{CarmonaCrepey}. The results in Table \ref{table:homB5} are for a model with moderate contagion, in-between the cases of independent obligors ($b=0$) and extreme contagion ($b=13$). For $b=0$, the algorithm in \cite{CarmonaCrepey} works well, however for models with moderate contagion it performs poorly even in the one-dimensional case, whereas the method presented here is asymptotically optimal. 
 
  \subsubsection*{Inhomogeneous groups with moderate contagion}
We now move to experiments with $d=2$ and $d=3$ in \eqref{eq:Model}, i.e., in the credit risk context, we allow there to be more than one group of obligors with different default intensities. The importance sampling results are obtained using the subsolution $\bar W $ described in Section \ref{sec:choiceDist}. Note that we have not attempted to optimise the selection of the constant $c$ in the different computations and this could potentially drive the relative error down even further. 
 
As a first example, we consider two groups, $d=2$, with intensities $a_1 = 0.01, a_2=0.05$ and moderate contagion $b=5$, with the first group constituting $80\%$ of the population; numerical results are presented in Table \ref{table:inhomB5}. 
\begin{table}[ht]
\caption{Importance sampling, using $\bar W $, and Monte Carlo estimates for inhomogeneous groups; $d=2$, \PN{$n=125$, $T=2$}, $a=(0.01,0.05)$, $b=5$, $w =( 0.8, 0.2)$.}
\label{table:inhomB5}
\centering
\begin{tabular}{c | c | c | c | c}
 & \multicolumn{2}{c|}{Importance sampling} & \multicolumn{2}{c}{Monte Carlo} \\
 $z$ & $\widehat p_n$ & $\textrm{RE}(\widehat p_n)$ & $\widehat p_n$ & $\textrm{RE}(\widehat p_n)$ \\ 
 \hline
 0.08 & 0.0331 & 0.0218 & 0.0328 & 0.0738 \\
 0.10 & 2.923e-3 & 0.0271 & 3.002e-3 & 0.258 \\
 0.12 & 4.592e-4 & 0.0354 & 4.700e-4 & 0.678 \\
 0.14 & 2.167e-5 & 0.0457 & 2.200e-5 & 2.859 \\
 0.16 & 2.457e-6 & 0.0460 & - & - \\
 0.20 & 7.094e-9 & 0.0570 & - & - \\
 0.24  & 1.382e-11 & 0.0800 & - & - \\
 0.28  & 2.077e-14 & 0.102 & - & - \\
 \hline
\end{tabular}
\end{table}
Although the algorithm is not provably optimal in this inhomogeneous case, it exhibits excellent performance, with a relative error of only about 10\% for probabilities of order $10^{-14}$.


Next, we consider three groups, $d=3$, with intensities $a_1=0.005$, $a_2=0.01$, and $a_3=0.05$, and moderate contagion ($b=5$). The two first groups constitute 40\% of the population each and the third group, with intensity $a_3 = 0.05$, constitutes the remaining 20\%; the numerical results are presented in Table \ref{table:inhomd3B5}.
\begin{table}[ht]
\PN{
\caption{\PN{Importance sampling, using $\bar W $, and Monte Carlo estimates for inhomogeneous groups; $d=3$, $n=125$, $T=2$, $a=(0.005, 0.01,0.05)$, $b=5$, $w =(0.4, 0.4, 0.2)$.}}
\label{table:inhomd3B5}
\centering
\begin{tabular}{c | c | c | c | c}
 & \multicolumn{2}{c|}{Importance sampling} & \multicolumn{2}{c}{Monte Carlo} \\
 $z$ & $\widehat p_n$ & $\textrm{RE}(\widehat p_n)$ & $\widehat p_n$ & $\textrm{RE}(\widehat p_n)$ \\ 
 \hline
 0.04 & 0.406 & 0.0131 & 0.406 & 0.0156 \\
 0.08 & 1.534e-2 & 0.0267 & 1.511e-2 & 0.119 \\ 
 0.12 & 1.078e-4 & 0.0432 & 9.600e-5 & 1.340 \\
 0.14 & 3.272e-6 & 0.0445 & 2.000e-6 & 10.000  \\
 0.16 & 2.708e-7 & 0.0547 & – & – \\
 0.20 & 3.363e-10 & 0.0686 & – & – \\
 0.24 & 2.631e-13 & 0.107 & – & – \\
 0.28 & 1.581e-16 & 0.211 & – & – \\
 \hline
\end{tabular}}
\end{table}

Similar to the results in Table \ref{table:inhomB5}, the importance sampling results in Table \ref{table:inhomd3B5}, based on the subsolution $\bar W$, are excellent for the range of $z$ considered here. Compared to the homogeneous setting—see Tables \ref{table:homB0}-\ref{table:homd5B5}—the relative error starts to slowly increase as the probability becomes smaller. This is to be expected, as $\bar W$ does not correspond to an asymptotically optimal method. However the improvement over both standard Monte Carlo and the methods presented in \cite{CarmonaCrepey} is substantial: even for probabilities of order $10^{-16}$ the observed relative error is below $1/4$. Together with the other results in this sections, this shows how the subsolution approach to importance sampling in general, and the proposed method in particular, can lead to significant improvement in the design of efficient algorithms compared to a ``naive'' state-independent change-of-measure. 


}

\begin{acks}
We thank the Associate Editor and anonymous referees for helpful feedback and suggestions, which helped improve the manuscript. 

B.\ Djehiche's research was supported by the Swedish Research Council, H.\ Hult's research was supported by the G\"oran Gustafsson Foundation, and P.\ Nyquist's research was supported by the Verg foundation, the Swedish Research Council (VR-2018-07050) and the Wallenberg AI, Autonomous Systems and Software Program (WASP) funded by the Knut and Alice Wallenberg Foundation. We are very grateful for their financial support.
\end{acks}





\bibliographystyle{ACM-Reference-Format}
\bibliography{references}

\appendix
\section{Derivation of Isaacs equation}
\label{sec:Isaacs}
For completeness, and for the reader who wishes to develop some intuition, we end the paper with a formal derivation of the Isaacs equation associated with the importance sampling estimator \eqref{eq:estimator}.
%
Naturally, the argument follows closely the general steps used in other works on the subsolution approach for dynamic importance sampling; see \cite{DupuisWangSubsol} for an overview. We emphasize that the derivation is of a formal nature and not all steps are motivated rigorously. 

\PN{Recalling the discussion in Section \ref{sec:IS}, the main quantity of interest for performance analysis is the second moment
\begin{align*}
	\E _{\bar{\bQ} ^n} \left[ \widehat p _n ^2 \right] = \mathbb{E} _{\bQ ^n} \left[ I\{N^{z} < N^{0} \} \prod_{k=1}^{N^{z}} \frac{\Theta ^n (d
	\tau_{k},v_{k}\mid X ^n (T_{k-1}))}{\bar \Theta ^n (d\tau_{k},v_{k}\mid X^n (T_{k-1}))} \right].
\end{align*}
To find the most efficient choice of $\bar \Theta ^n$, one can minimise this second moment with respect to $\bar \Theta ^n$: Define $V^n$ to be this optimal second moment,

\begin{align*}
	V^n = \inf _{\bar \Theta ^n} E _{\Theta ^n} \left [ I\{N^{z} < N^{0}\} \prod_{k=1}^{N^{z}} \frac{\Theta^n(d
	\tau_{k},v_{k}\mid X^{n}(T_{k-1}))}{\bar \Theta^n(d\tau_{k},v_{k}\mid X^{n}(T_{k-1}))} \right],
\end{align*}
and define $W^n$ through a large-deviation-type scaling,
\begin{align*}
	W^n = -\frac{1}{n}\log V^n,
\end{align*}
To employ tools from stochastic control, for both $V^n$ and $W^n$ one can define time-and-state-dependent analogues by conditioning:
\begin{align}
\label{eq:VnApp}
	V^{n}(t,\bx) = \inf_{\bar \Theta^{n}} E_{\Theta ^n}\left[ I\{N^{z} < N^{0}\} \prod_{k=l}^{N^{z}} \frac{\Theta ^n (d
	\tau_{k},v_{k}\mid X^{n}(T_{k-1}))}{\bar \Theta ^n(d\tau_{k},v_{k}\mid X^{n}(T_{k-1}))} \mid X^{n}(t) = \bx\right],
\end{align}
where $l$ is such that $T_{l-1} \leq t < T_{l}$, and $W^n(t,\bx)$ is defined from $V^n(t,\bx)$ as $W^n$ is from $V^n$.} By the memoryless property of the exponential distribution no correction for the first step is needed in \eqref{eq:VnApp}; $\tau _l$ and $\tau_{l}\mid \tau_{l} \geq t-T_{l-1}$ have the same distribution.

As described in Section \ref{sec:IS} the jump intensities under consideration are of the form $	\bar r ^n (\bx,\be _j) = n \bar \lambda _j(\bx)$, where $\bar \lambda(\bx) = (\bar \lambda _1(\bx) , \dots , \bar \lambda _d (\bx))$ is not identical to the zero element in $\R ^d$. Since each stochastic kernel $\bar \Theta ^n$ is determined by the corresponding $\bar \lambda$, the infimum in \eqref{eq:VnApp} is over those $\bar \lambda(\bx)$ that are zero only for directions $j$ for which $\lambda$ is zero. Note that there will be a slight abuse of notation in that supremum and infimum is taken over $\bar r^n$, and in including the argument $\bx$ although the optimization will always take place for a fixed state $\bx$.

Because the process $X^n$ is constant between jumps, the likelihood ratio in \eqref{eq:VnApp} can be expressed as
\begin{align*}
	 &\textrm{exp} \left \{ \int _{T_l} ^{T _{N^z}} (\bar R (X^n(s)) - R(X^n (s)))ds + \sum _{k=l} ^{N^z} \log   \frac{r ^n (X ^n (T_{k-1}), v_k)}{\bar r ^n (X ^n (T_{k-1}), v_k)}  \right\} \\
	& \quad =  \textrm{exp} \left \{ \sum _{k=l} ^{N^z} \Bigl( (\bar R (X^n(T_{k-1})) - R( X ^n (T_{k-1})) ) \tau _k \right. \\
	&  \left. \left.  \quad \qquad \qquad+ \log  \frac{r ^n (X ^n (T_{k-1}), v_k)}{\bar r ^n (X ^n (T_{k-1}), v_k)}   \right)  \right \}.
\end{align*}
Integration over the joint distribution of $(\tau _l , v_l)$, using the independence of the $\tau _k$'s, yields
\begin{align*}
	V^n (t,\bx) &= \inf_{\bar r^{n} (\bx,\cdot)} \int _0 ^{\infty} \sum _{j=1} ^d \textrm{exp} \left \{ (\bar R (\bx) - R(\bx))u + \log \frac{r^n (\bx,\be _j)}{\bar r ^n (\bx, \be _j)} \right \} \\
	& \quad \times E_{\Theta ^n} \left[ I\{N^{z} < N^{0}\}  \textrm{exp} \left \{ \int _{T_{l+1}} ^{T_{N^z}}  \left( \bar R (X^n(s)) - R( X ^n (s)) \right) ds \right. \right. \\
	& \left. \left. \qquad \quad \qquad + \sum _{k =l+1} ^{N^z} \log  \frac{r ^n (X ^n (T_{k-1}), v_k)}{\bar r ^n (X ^n (T_{k-1}), v_k)}   \right \}   \mid X^{n}(t+u) = \bx + \frac{\be_j}{n}\right]  \\
	& \quad \times r^n (\bx,\be _j) e^{-R(\bx)u}du,
\end{align*}
and by the dynamic programming principle $V^n (t,\bx)$ satisfies the following dynamic programming equation,
\begin{align}
\begin{split}
\label{eq:dpeVn}
	V^n (t,\bx) = \inf_{\bar r^{n} (\bx,\cdot)} & \Biggl \{ \int _0 ^{\infty} \sum _{j=1} ^d e^{ (\bar R(\bx) - R(\bx))u } \frac{r^n (\bx,\be_j)}{\bar r ^n (\bx, \be_j)} \\
	&  \quad \times V^n (t+u,\bx+ \frac{\be_j}{n}) r ^n (\bx,\be_j) e^{-R(\bx)u}du \Biggr \}.
	\end{split}
\end{align}
From equation \eqref{eq:dpeVn} for $V^n(t,\bx)$ we obtain the corresponding equation for $W^n(t,\bx)$,
\begin{align*}
	n W^n (t,\bx) = \sup_{\bar r^{n} (\bx,\cdot)} & \Biggl \{ - \log \int _0 ^{\infty} \sum _{j=1} ^d e^{ (\bar R(\bx) - R(\bx))u } \frac{r^n (\bx,\be_j)}{\bar r ^n (\bx, \be_j)}\\
	& \quad \times e^{-n W^n (t+u, x+ \frac{e_j}{n})} r ^n (x,e_j) e^{-R(x)u}du \Biggr \}.
\end{align*}
Let $\hat \Theta ^n $ denote the stochastic kernel based on a set of jump intensities $\hat r ^n (\bx,\cdot)$, where $\hat r ^n (\bx,\be_j) = n \hat \lambda _j (\bx)$ for some function $\hat \lambda \colon D \to \R ^d$. 
From the relative entropy representation for exponential integrals (see, e.g., \cite{Dupuis97, DupuisWangSubsol}), 
\begin{align*}
	n W^n (t,\bx) &= \sup_{\bar r^{n} (\bx,\cdot)} \inf_{\hat \Theta ^{n} } \biggl \{ \re(\hat \Theta ^n \mid \Theta ^n )+ \int _0 ^{\infty} \sum _{j=1} ^d \bigl( ( R(\bx) - \bar R(\bx))u \\
	& \quad + \log \frac{\bar r^n (\bx,\be_j)}{r ^n (\bx, \be_j)} + n W^n (t+u, \bx+ \frac{\be_j}{n}) \bigr) \hat r ^n (\bx,\be_j) e^{-\hat R(\bx)u}du \biggr \},
\end{align*}
where $\re$ denotes the relative entropy. The infimum over $\hat \Theta ^n$ is equivalent to infimum over jump intensities $\hat r ^n (\bx,\cdot)$ and the likelihood ratio between $\hat \Theta ^n$ and $\Theta ^n$ is of the same form as the likelihood ratio between $\bar \Theta ^n$ and $\Theta ^n$. Writing out the relative entropy term explicitly and moving $nW^n(t,\bx)$ to the right-hand side,
 \begin{align}
 	0 &= \sup_{\bar r^{n} (\bx,\cdot)} \inf_{\hat r ^{n}(\bx,\cdot) } \biggl \{ \int _0 ^{\infty} \sum _{j=1} ^d \Big( (2R(\bx) - \bar R(\bx) - \hat R(\bx))u \label{eq:expression1} \\
	& \quad + \log \bar r ^n (\bx,\be_j) + \log \hat r ^n (\bx,\be_j) - 2 \log r^n (\bx,\be_j) \label{eq:expression2} \\
	& \quad + n\bigl( W^n (t+u, \bx + \frac{\be_j}{n}) - W^n (t,\bx) \bigr)  \Big) \hat r^n(\bx,\be_j) e^{-\hat R (\bx) u }du \label{eq:expression3} \biggr \}.
 \end{align}
 The three terms on the right-hand side are treated separately. The first two integrals are straightforward to compute:
 \begin{align*}
 	&\int _0 ^{\infty} \sum _{j=1} ^d \left( 2R(\bx) - \bar R(\bx) - \hat R(\bx) \right)u \hat r^n(\bx,\be_j) e^{-\hat R (\bx) u }du \\
	& \quad = \frac{1}{\hat R (\bx)} \left( 2R(\bx) - \bar R(\bx) - \hat R(\bx) \right),
 \end{align*}
and
\begin{align*}
 	& \int _0 ^{\infty} \sum _{j=1} ^d \bigl( \log \bar r ^n (\bx,\be_j) + \log \hat r ^n (\bx,\be_j) - 2 \log r^n (\bx,\be_j) \bigr)  \hat r^n(\bx,\be_j) e^{-\hat R (\bx) u }du \\
	& \quad = \frac{1}{\hat R (\bx)} \sum _{j=1} ^d \hat r^n (\bx,\be_j) \left( \log \bar r ^n (\bx,\be_j) + \log \hat r ^n (\bx,\be_j) - 2 \log r^n (\bx,\be_j) \right).
 \end{align*} 
The third integral, with integrand given in \eqref{eq:expression3}, can be expressed as an expectation involving an exponentially distributed random variable. Indeed, let $\{ \xi _n \}$ be a sequence of random variables each having an exponential distribution with mean $\hat R (\bx) ^{-1}$. Then,
 \begin{align*}
	& \int _0 ^{\infty} \sum _{j=1} ^d n \left (W^n (t+u, \bx + \frac{\be_j}{n}) - W^n (t,\bx) \right) \hat r^n(\bx,\be_j) e^{-\hat R (\bx) u }du \\
	& \quad = \sum _{j=1} ^d \frac{\hat r^n (\bx,\be_j)}{\hat R(\bx)} E \left[ n(W^n(t + \xi _n, \bx+\frac{\be_j}{n}) - W^n (t,\bx)) \right].
 \end{align*}
The expression involving (the integrals of) \eqref{eq:expression1}-\eqref{eq:expression3} can then be rewritten as
 \begin{align*}
 	0 &= \sup_{\bar r^{n} (\bx,\cdot)} \inf_{\hat r ^{n}(\bx,\cdot) } \Biggl \{ \frac{1}{\hat R (\bx)} ( 2R(\bx) - \bar R(\bx) - \hat R(\bx) ) \\
	& \quad + \frac{1}{\hat R (\bx)} \sum _{j=1} ^d \hat r^n (\bx,\be_j) \left( \log \bar r ^n (\bx,\be_j) + \log \hat r ^n (\bx,\be_j) - 2 \log r^n (\bx,\be_j) \right) \\
	& \quad + \sum _{j=1} ^d \frac{\hat r^n (\bx,\be_j)}{\hat R(\bx)} E \left[ n(W^n(t + \xi_n, \bx+\frac{\be_j}{n}) - W^n (t,\bx)) \right] \Biggr \}.
 \end{align*}
 Define the function $l \colon \R \to [0,\infty]$ by
 \begin{align*}
 	l(x) = \begin{cases} x\log x -x +1, & x \geq 0 \\ \infty, & \textrm{otherwise}. \end{cases}
 \end{align*}
With $\xi _n \sim \textrm{Exp} (n \hat \Lambda (\bx))$ and using the same notation for the jump intensities $\hat r^n$ and $\bar r^n$ (i.e., $\hat \lambda, \hat \Lambda$ and $\bar \lambda, \bar \Lambda$), the equation of interest can be expressed as
\begin{align*}
0 = \sup_{\bar \lambda (\bx)} \inf_{\hat \lambda (\bx) } & \Biggl \{ \frac{1}{\hat \Lambda (\bx)} \biggl( 2 \sum _{j=1} ^d \lambda _j(\bx) l \bigl ( \frac{\hat \lambda _j (\bx)}{\lambda _j (\bx)}\bigr) - \sum _{j=1} ^d \bar \lambda _j(\bx) l \bigl( \frac{\hat \lambda _j (\bx)}{\bar \lambda _j (\bx)}\bigr) \biggr) \\
&  \quad + \frac{1}{ \hat \Lambda(\bx)} \sum _{j=1} ^d \hat \lambda_j (\bx) E \left[ n(W^n(t + \xi _n, \bx+\frac{\be_j}{n}) - W^n (t,\bx)) \right] \Biggr \}.
\end{align*}
Recall that $W_t$ and $\gradx W$ denotes the time derivative of $W$ and the gradient in the space variable $\bx$, respectively. To formally obtain a limit PDE related to the stochastic control problem, assume that there is a suitable limit $W$ for $W^n$. More precisely, that there is a smooth function $W$ such that, as $n \rightarrow \infty$, $W^n (t,\bx) \to W(t,\bx)$ and
\begin{align*}
	n (W^n (t + \frac{u}{n}, \bx + \frac{\be_j}{n}) - W^n (t,\bx)) \rightarrow u W_t (t,\bx) + \langle \gradx W (t,\bx), \be_j \rangle.
\end{align*}
Consider the expectation involving the $\xi _n$'s. By a change of variable,
\begin{align*}
	&E \left[ n(W^n(t + \xi _n, \bx+\frac{e_j}{n}) - W^n (t,\bx)) \right] \\
	& \quad= \int _0 ^{\infty} n \left( W^n (t + \xi, \bx + \frac{e_j}{n}) - W^n(t,\bx) \right) n \hat \Lambda (\bx) e ^{-n \hat \Lambda(\bx) \xi} d\xi \\
	& \quad = \int _0 ^{\infty} n \left( W^n (t + \frac{\tau}{n}, \bx + \frac{\be_j}{n}) - W^n(t,\bx) \right) \hat \Lambda (\bx) e ^{- \hat \Lambda(\bx) \tau} d\tau.
\end{align*}
As $n$ goes to infinity, the assumed convergence of $W^n$ implies that, for each $\tau$, the integrand converges to $\tau W_t(t,\bx) + \langle \gradx W(t,\bx),\be_j \rangle$. Taking the limit inside the expectation,
\begin{align*}
	\lim _{n \rightarrow \infty} E \left[ n(W^n(t + \xi _n, \bx+\frac{\be_j}{n}) - W^n (t,\bx)) \right] = \frac{W_t (t,\bx)}{\hat \Lambda (\bx)} + \langle \gradx W(t,\bx),\be_j \rangle.
\end{align*}
Thus, in the limit as $n$ goes to infinity, the dynamic programming equation for $W^n$ gives rise to the following Isaacs equation,
 \begin{align}
 \begin{split}
 \label{eq:IsaacsAlt}
 	0 = \sup_{\bar \lambda (\bx)} \inf_{\hat \lambda (\bx) } \frac{1}{\hat \Lambda (\bx)} & \Biggl \{ \sum _{j=1} ^d \Bigr( 2\lambda _j(\bx) l \bigl ( \frac{\hat \lambda _j (\bx)}{\lambda _j (\bx)}\bigr) -  \bar \lambda _j(\bx) l \bigl( \frac{\hat \lambda _j (\bx)}{\bar \lambda _j (\bx)}\bigr) \Bigr)  + W_t (t,\bx) \\ 
	& \quad + \sum _{j=1} ^d \hat \lambda_j (\bx) \langle DW(t,\bx),\be_j \rangle \Biggr \}.
	\end{split}
 \end{align}
Note that, aside from the time derivative $W_t (t, \bx)$, this is equation (6.4) in \cite{DupuisWangLeder_wslqIS} (with $\nabla W(x)$ replaced by $\gradx W(t,\bx)$). Define the Hamiltonian $\mathbb{H}$ on $ D \times \R ^d$ by
\begin{align*}
	\mathbb{H} (\bx,\alpha) = \sup_{\bar \lambda (\bx)} \inf_{\hat \lambda (\bx) } \Biggl \{ \sum _{j=1} ^d \Bigr( 2\lambda _j(\bx) l \bigl ( \frac{\hat \lambda _j (\bx)}{\lambda _j (\bx)}\bigr) -  \bar \lambda _j(\bx) l \bigl( \frac{\hat \lambda _j (\bx)}{\bar \lambda _j (\bx)}\bigr) + \hat \lambda_j \b(x) \langle \alpha ,\be_j \rangle\Bigr) \Biggr \}.
\end{align*} 
Recall the definition \eqref{eq:Hamiltonian} of the Hamiltonian $H$,
\begin{align*}
	H(\bx,\alpha) = \sum _{j=1} ^d \lambda _j(\bx) \left( e^{\langle \alpha, \be_j \rangle} - 1\right).
\end{align*}
The following result from \cite{DupuisWangLeder_wslqIS} characterizes saddle points of $\mathbb{H}$.  \begin{proposition}[Proposition 6.2 in  \cite{DupuisWangLeder_wslqIS}]
 \label{prop:Hamiltonian}
 	For any $\bx \in D$ and $\alpha \in \R ^d$,
	\begin{align*}
		\mathbb{H}(\bx,\alpha) = -2 H (\bx, -\frac{\alpha}{2}),
	\end{align*}
	and the saddle point for $\mathbb{H}$ is given by $(\bar \lambda, \hat \lambda)$ such that
	\begin{align*}
		\bar \lambda _j (\bx) = \hat \lambda _j (\bx) = \lambda _j (\bx) e ^{- \frac{\langle \alpha, \be_j \rangle}{2}}.
	\end{align*}
 \end{proposition}
As mentioned in \cite{DupuisWangLeder_wslqIS}, from the existence of saddle points for $\mathbb{H}$ one can argue that the factor $\hat \Lambda (\bx) ^{-1}$ in \eqref{eq:IsaacsAlt} can be removed. Indeed, the time derivative does not change this and the Isaacs equation \eqref{eq:IsaacsAlt} becomes
 \begin{align}
 \label{eq:Isaacs}
 	W_t (t, \bx) + \mathbb{H}(\bx, \gradx W(t,\bx)) = 0.
 \end{align}
From the definition of the value function $V^n (t,\bx)$ it is clear that $W^n(t,\bx)$ must satisfy the terminal condition
\begin{align*}
	W^n (T,\bx) =  \begin{cases} 0, & \bx \in D_z, \\
	\infty, & \textrm{otherwise}, \end{cases}
\end{align*}
which in turn carries over to the function $W$. Proposition \ref{prop:Hamiltonian} combined with this terminal condition implies that the Isaacs equation \eqref{eq:Isaacs} is indeed the Hamilton-Jacobi equation
 \begin{align}
  \label{eq:IsaacsFinalApp}
  \begin{cases}
   	W_t (t,\bx) - 2 H \left(\bx, - \frac{1}{2}\gradx W(t,\bx) \right) = 0 & (t,\bx) \in [0,T) \times D \setminus D_z, \\
	W(T,\bx) = 0, & \bx \in D_z.
	\end{cases}
 \end{align}

\end{document}